\documentclass[11pt,reqno]{amsart}

\usepackage[T1]{fontenc}
\pdfoutput=1
\usepackage{mathptmx,microtype,hyperref,graphicx,dsfont}
\usepackage{amsthm,amsmath,amssymb}
\usepackage[foot]{amsaddr}
\usepackage{enumitem}
\usepackage[nameinlink,capitalise]{cleveref}
\usepackage[font=footnotesize,margin=2em]{caption}
\usepackage{cite}

\theoremstyle{plain}
\newtheorem{theorem}{Theorem}
\newtheorem{lemma}[theorem]{Lemma}
\newtheorem{claim}[theorem]{Claim}

\newcommand\cS{\mathcal{S}}
\newcommand{\cT}{{\mathcal T}}

\newcommand\R{\mathbb{R}}

\renewcommand{\P}{{\mathbb P}}
\newcommand\E{{\mathbb E}}

\newcommand{\eps}{\varepsilon}
\renewcommand{\le}{\leqslant}
\renewcommand{\ge}{\geqslant}

\newcommand{\1}{{\bf 1}}

\author[B. Kolesnik]{Brett Kolesnik}
\address{University of Warwick, Department of Statistics}
\email{brett.kolesnik@warwick.ac.uk}

	%Taxonomy, cladistics, statistics in mathematical biology}

\begin{document}

\title[Critical beta-splitting, via contraction]
{Critical beta-splitting, via contraction}

%%%%%%%%%%%%%%%%%%%%%%%%%%%%
%%%%%%%%%%%%%%%%%%%%%%%%%%%%
%%%%%%%%%%%%%%%%%%%%%%%%%%%%
%%%%%%%%%%%%%%%%%%%%%%%%%%%%
%%%%%%%%%%%%%%%%%%%%%%%%%%%%
\begin{abstract} 
The critical beta-splitting tree, introduced by Aldous, 
is a Markov branching phylogenetic tree.   
Aldous and Pittel recently proved, amongst other results, 
a central limit theorem for the 
height 
of a random leaf. 
We give an alternative proof, 
via contraction methods for 
random recursive structures. 
These methods were developed by 
Neininger and R\"{u}schendorf, 
motivated by 
 Pittel's article 
``Normal convergence problem? 
Two moments and a recurrence may be the clues.''
Aldous and Pittel estimated the leading order terms
in the first two moments. More recently, 
Aldous and Janson 
obtained an asymptotic expansion for the average height. 
We show that a central limit theorem follows, 
and bound the distance to normality.  
Our results also apply to the continuous version of the model, 
in which branching times are exponential. 

\end{abstract}

\maketitle

%%%%%%%%%%%%%%%%%%%%%%%%%%%%
%%%%%%%%%%%%%%%%%%%%%%%%%%%%
%%%%%%%%%%%%%%%%%%%%%%%%%%%%
%%%%%%%%%%%%%%%%%%%%%%%%%%%%
%%%%%%%%%%%%%%%%%%%%%%%%%%%%

\begin{figure}[h]
\centering
\includegraphics[scale=1]{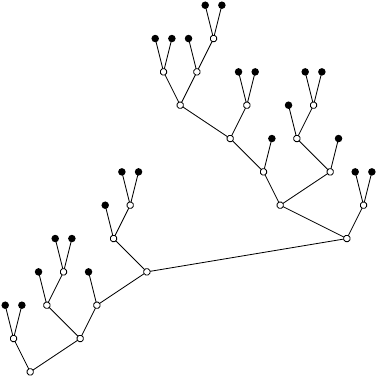}
\caption{A critical beta-splitting tree with leaves
labelled by $\{1,\ldots,23\}$. 
The location of an internal node 
indicates when and where 
the 
set of leaves above it was split into two subsets. 
Initially, $\{1,\ldots,23\}$ splits 
into $\{1,2\}$ and $\{3,\ldots,23\}$.
Then, $\{1,2\}$ splits into $\{1\}$ and $\{2\}$;  
$\{3,\ldots,23\}$ splits into 
$\{3,4,5\}$ and $\{6,\ldots,23\}$; etc., 
until only the singleton sets $\{1\},\ldots,\{23\}$ remain. 
}
\label{F_Tn}
\end{figure}

\newpage

%%%%%%%%%%%%%%%%%%%%%%%%%%%%
%%%%%%%%%%%%%%%%%%%%%%%%%%%%
%%%%%%%%%%%%%%%%%%%%%%%%%%%%
%%%%%%%%%%%%%%%%%%%%%%%%%%%%
%%%%%%%%%%%%%%%%%%%%%%%%%%%%
\section{Introduction}\label{S_intro}

The 
{\it critical beta-splitting tree} $\cT_n$, 
introduced by Aldous \cite{Ald96}, 
is a random recursive combinatorial structure, 
constructed as follows: 

We begin with the set $\{1,\ldots,n\}$. (We assume that $n\ge2$, 
as otherwise $\cT_1$ is a single vertex.) 
Let 
\begin{equation}\label{E_vartheta}
\vartheta(n) = \sum_{i=1}^n \frac{1}{i}\sim\log n
\end{equation}
denote the harmonic sum. 
The first {\it split} occurs between some $i$ and $i+1$
with probability 
\begin{equation}\label{E_pni}
p(n,i)=
\frac{1}{2\vartheta (n-1)}
\frac{n}{i(n-i)},
\quad\quad
1\le i\le n-1, 
\end{equation}
in which case $\{1,\ldots,n\}$ separates  
into $\{1,\ldots,i\}$
and $\{i+1,\ldots,n\}$.  
We continue recursively, splitting 
$\{1,\ldots,i\}$
and $\{i+1,\ldots,n\}$
independently, etc., until 
only the singleton sets $\{1\},\ldots,\{n\}$ remain. 

The tree $\cT_n$ is constructed using the set $\cS$ of subsets of 
$\{1,\ldots,n\}$ determined by the above splitting procedure: 
For each $1\le j\le n$, a leaf $v(j)=v(\{j\})$ is placed at {\it height} 
\[
h(j)=\#\{S\in\cS:j\in S\}-1. 
\]
An {\it internal node} $v(S)$ is added
to the tree, for each $S\in\cS$ with $\#S>1$. 
The two children of $v(S)$ are $v(S_1)$
and $v(S_2)$, 
where $S_1,S_2\in\cS$ are the unique
pair for which $S=S_1\cup S_2$. 
See Figure  \ref{F_Tn} for an example. 

There are $n-1$ internal nodes in total, 
one between 
each $i$ and $i+1$. 
The first internal node $\rho=v(\{1,\ldots,n\})$
is called the {\it root} of $\cT_n$. 
The ``$-1$'' in the formula for $h(j)$ 
accounts for the singleton set $\{j\}\in\cS$, 
which does not contribute to the height of $v(j)$. 
Indeed, the height $h(j)$ is simply the graph distance 
(number of edges) 
between $\rho$ and $v(j)$. 

As discussed in \cite{Ald96}, 
the tree $\cT_n$ is ``critical'' in the following sense: 
A tree could be constructed in a similar way, 
but with $p(n,i)$
proportional 
to  $i^\beta(n-i)^\beta$. 
The value $\beta=-1$ in the above construction 
is of particular interest, since, at this point, 
typical heights $h(j)$ 
switch from polynomial 
to poly-logarithmic order.

%%%%%%%%%%%%%%%%%%%%%%%%%%%%
%%%%%%%%%%%%%%%%%%%%%%%%%%%%
%%%%%%%%%%%%%%%%%%%%%%%%%%%%
%%%%%%%%%%%%%%%%%%%%%%%%%%%%
%%%%%%%%%%%%%%%%%%%%%%%%%%%%
\subsection{Results} 

Amongst other results, 
Aldous and Pittel \cite{AP23} 
recently proved a central limit theorem for 
$H_n=h(J_n)$, 
where $J_n$ is uniformly random in $\{1,2,\ldots,n\}$
and independent of the tree $\cT_n$. 
In other words, $H_n$ is the height 
of a random leaf in $\cT_n$. 

\begin{theorem}[{Aldous and Pittel \cite[Theorem 1.7]{AP23}}]
\label{T_CLT}
As $n\to\infty$, 
we have that 
\[
\frac{H_n-\frac{1}{2\zeta(2)}\log^2n}
{\sqrt{\frac{2\zeta(3)}{3\zeta(2)^3}\log^3 n}}
\]
is asymptotically standard normal distributed. 
\end{theorem}

In the above, as usual,  
for $s=a+ib$ with $a>1$, 
\[
\zeta(s)=\sum_{n=1}^\infty \frac{1}{n^s}
\]
is the 
{\it Riemann zeta function}. 
For convenience, we write 
$\log^\gamma n=(\log n)^\gamma$. 

Many other 
features
of the beta-splitting
tree have been analyzed 
recently by 
Aldous and Pittel \cite{AP23}; 
Aldous and Janson \cite{AJ23,AJ24a,AJ24b}; 
Aldous, Janson and Li \cite{AJL24};  and 
Iksanov \cite{Iks24,Iks24b}.

%%%%%%%%%%%%%%%%%%%%%%%%%%%%
%%%%%%%%%%%%%%%%%%%%%%%%%%%%
%%%%%%%%%%%%%%%%%%%%%%%%%%%%
%%%%%%%%%%%%%%%%%%%%%%%%%%%%
%%%%%%%%%%%%%%%%%%%%%%%%%%%%
\subsection{Purpose}
\label{S_purpose}

Our purpose is to give an alternative 
proof of the asymptotic normality of
$H_n$. 
We will use the contraction methods 
of 
Neininger and R\"{u}schendorf \cite{NR04a,NR04b} (cf.\ 
\cite{Ros91,Ros92,RR95,RR01}), 
together with the 
estimates for the mean and variance of $H_n$  
obtained by Aldous and Pittel \cite{AP23},  and 
further refinements 
by Aldous and Janson \cite{AJ24b} 
(see Section \ref{S_moments}). 
In fact, as we will see, the beta-splitting tree
is a borderline case, where the 
contraction results (with trivial fixed point) proved in 
\cite{NR04b} do not apply as stated. 
We will revisit and adapt the proofs
to deal with the situation at hand.  

The contraction method also applies to the continuous
version of the model (see Section \ref{S_ctsT}), in which branching
times are exponential, yielding 
an alternative proof of the central 
limit theorems proved by various other means in  
\cite[Theorem 1.7]{AP23}, 
\cite[Theorem 2]{AJ23}
and 
\cite[Theorem 1.5]{AJ24b}
for the {\it time-height} $\hat H_n$
of a random leaf. 
One advantage of the contraction method is that,  
once the first two moments have been estimated,  
essentially the same proof establishes the asymptotic 
normality of $H_n$
and $\hat H_n$, 
along with a bound 
on the
distance 
(order $\log^{-1/2+\eps} n$, 
with respect to the 
Zolotarev metric)
from normality; 
see Theorems  
\ref{T_CLT2} and \ref{T_CLThatH}.

%%%%%%%%%%%%%%%%%%%%%%%%%%%%
%%%%%%%%%%%%%%%%%%%%%%%%%%%%
%%%%%%%%%%%%%%%%%%%%%%%%%%%%
%%%%%%%%%%%%%%%%%%%%%%%%%%%%
%%%%%%%%%%%%%%%%%%%%%%%%%%%%
\subsection{Discussion}
\label{S_discussion}

Let us also mention a 
connection with the number of collisions in the $\beta(2,b)$-coalescent, 
studied by 
Iksanov, Marynych and M\"{o}hle \cite{IMM09}. 
For similar reasons, 
\cite{NR04b}
does not directly apply in this setting. 
In \cite{IMM09}, this obstacle is overcome 
by identifying a more complicated recursive 
property of the $\beta(2,b)$-coalescent (see (14) in \cite{IMM09}) 
for which \cite{NR04b} applies.  
However, they ask \cite[Remark 1.6]{IMM09}
if the central limit theorem \cite[Theorem 1.5]{IMM09}
can be proved more directly. 
Our current strategy 
suggests a possible route towards this open question.
See Section \ref{S_Collisions} for more details.

Finally, let us remark that 
Oleksandr Iksanov has 
informed us, by private communication, that 
Theorem \ref{T_CLT} can be derived 
by Theorem 12 in Gnedin, Pitman and Yor \cite{GPY06} 
and by Theorem 3.1(a) in Gnedin and Iksanov \cite{GI12}. 
In fact, Iksanov \cite{Iks24b} 
has established the joint
convergence of $(H_n,\hat H_n)$, 
via 
regenerative compositions \cite{GP05,GI12}.

%%%%%%%%%%%%%%%%%%%%%%%%%%%%
%%%%%%%%%%%%%%%%%%%%%%%%%%%%
%%%%%%%%%%%%%%%%%%%%%%%%%%%%
%%%%%%%%%%%%%%%%%%%%%%%%%%%%
%%%%%%%%%%%%%%%%%%%%%%%%%%%%
\subsection{Acknowledgments}

This project began at the University
of Oxford, while we were
supported by a 
Florence Nightingale Bicentennial Fellowship 
at the Department of Statistics 
and a Senior Demyship at Magdalen College. 
We thank 
David Aldous, 
Oleksandr Iksanov, 
Svante Janson, 
Martin M\"{o}hle, 
Ralph Neininger and 
Boris Pittel 
for enlightening conversations,
and the anonymous referees 
for their careful reading 
and helpful comments.

%%%%%%%%%%%%%%%%%%%%%%%%%%%%
%%%%%%%%%%%%%%%%%%%%%%%%%%%%
%%%%%%%%%%%%%%%%%%%%%%%%%%%%
%%%%%%%%%%%%%%%%%%%%%%%%%%%%
%%%%%%%%%%%%%%%%%%%%%%%%%%%%
\section{Background}
\label{S_dis}

\subsection{Pittel's principle}

The limit theory in
\cite{NR04a} was, in part, developed in response
to the work of Pittel \cite{Pit99}, in which 
 limit theorems are proved for various combinatorial 
quantities of interest (e.g., the independence number 
of a uniformly random labelled tree) 
with mean and variance that are close to linear. 

The following line of reasoning is referred to as  
``Pittel’s principle'' in \cite[p.\ 379]{NR04a}. 
Indeed, in \cite[p.\ 1260]{Pit99}, the author states that:  
\begin{quote}
For various global characteristics of large size 
combinatorial structures [...]\
one can usually estimate the mean and the variance, 
and also obtain a recurrence 
for the generating function 
[...].
As a heuristic principle based on our experience, 
we claim that such a characteristic is asymptotically 
normal if the mean and the variance are “nearly linear” [...]. 
The technical reason is that in such a case the 
moment generating function [...]\
of the normal distribution 
with the same two moments “almost” satisfies the recurrence.
\end{quote}

%%%%%%%%%%%%%%%%%%%%%%%%%%%%
%%%%%%%%%%%%%%%%%%%%%%%%%%%%
%%%%%%%%%%%%%%%%%%%%%%%%%%%%
%%%%%%%%%%%%%%%%%%%%%%%%%%%%
%%%%%%%%%%%%%%%%%%%%%%%%%%%%
\subsection{Contraction theory}

A general theory is developed in \cite{NR04a}, 
which, in particular, yields limit theorems in such situations, 
as in \cite{Pit99}; 
see \cite[Corollary 5.2]{NR04a}. 
In fact, their results apply to a large family of 
random structures $X_n$, which satisfy a distributional 
recurrence of the form
\begin{equation}\label{E_NR}
X_n\overset{d}{=} \sum_{r=1}^K A_r(n)X^{(r)}_{I_r^{(n)}}+b_n. 
\end{equation}
As discussed in \cite{NR04a}, 
such situations arise, e.g., in divide-and-conquer
type algorithms. 
In this context, $b_n$ is called the ``toll function,'' 
associated with the ``cost'' of splitting into smaller, but similar,  
subproblems. 

Under certain conditions, 
a limit theorem can be proved
for $X_n$ satisfying \eqref{E_NR}, via the 
so-called 
{\it contraction method}. 
Roughly speaking, this strategy   
aims to identify the limiting distribution of $X_n$ 
by means of the fixed point equation 
\begin{equation}\label{E_NRfp}
X
\overset{d}{=} 
\sum_{r=1}^K A_r^* X^{(r)} +b^*, 
\end{equation}
obtained 
by taking $n\to\infty$ in \eqref{E_NR}.
The normal distribution is associated with  
the situation that 
$\sum_{i=1}^K (A_r^*)^2=1$ and 
$b^*=0$. 

See \cite[Theorem 5.1 and Corollary 5.2]{NR04a}
for univariate results. 
See also 
\cite[\S5.4]{NR04a}
for discussions on the multivariate case, 
when $K=K_n$ is random, and potentially even 
$K_n\to\infty$. 
Several applications are 
discussed in
\cite[\S5.2--5.3]{NR04a} and \cite[\S4--5]{NR04b}. 
In many cases, previous results  
follow more easily using these
techniques. 
We were introduced to these methods   
while studying 
randomized importance sampling algorithms 
for perfect matchings  
\cite{DK21} (cf.\ Neininger and Straub \cite{NS22}).

%%%%%%%%%%%%%%%%%%%%%%%%%%%%
%%%%%%%%%%%%%%%%%%%%%%%%%%%%
%%%%%%%%%%%%%%%%%%%%%%%%%%%%
%%%%%%%%%%%%%%%%%%%%%%%%%%%%
%%%%%%%%%%%%%%%%%%%%%%%%%%%%
\subsection{Beta-splitting recurrence}

The height $H_n$ of 
a random leaf in 
the critical beta-splitting tree $\cT_n$
satisfies a
simple recurrence of the form \eqref{E_NR}. 
Specifically, we claim the following. 

\begin{lemma}\label{L_rec}
We have that 
\begin{equation}\label{E_rec}
H_n
\overset{d}{=} 
H_{I_n}
+1, 
\end{equation}
where
\begin{equation}\label{E_In}
\P(I_n=i) = \frac{1}{(n-i)\vartheta (n-1)}, 
\quad\quad 1\le i\le n-1. 
\end{equation}
\end{lemma}

\begin{proof}
This follows by the construction
of $\cT_n$ and the symmetry of the splitting 
distribution $p(n,i)$. 
More precisely, consider the first split
in the construction of $\cT_n$, which adds 
the height ``$+1$'' in \eqref{E_rec} to the tree (see Figure \ref{F_Tn}). 
If the first split occurs between $i$ and $i+1$ (with probability $p(n,i)$)
or between $n-i$ and $n-i+1$ (with probability $p(n,n-i)$), 
then the rest of the tree $\cT_n$
is constructed using two independent trees, 
started from height $1$, of sizes $i$ and $n-i$.  The random leaf is in the 
tree of size $i$ with probability $i/n$. 
Therefore, since $p(n,i)=p(n,n-i)$ (see \eqref{E_pni}), 
it follows that the height of a random vertex in a 
tree of height $n$ is $1$ plus the height 
of a random vertex in a tree of height $1\le i\le n-1$, 
with probability 
\[
\frac{i}{n}(p(n,i)+p(n,n-i))
=\frac{2i}{n}p(n,i)=
\frac{1}{(n-i)\vartheta(n-1)},
\]
as claimed. 
\end{proof}

%%%%%%%%%%%%%%%%%%%%%%%%%%%%
%%%%%%%%%%%%%%%%%%%%%%%%%%%%
%%%%%%%%%%%%%%%%%%%%%%%%%%%%
%%%%%%%%%%%%%%%%%%%%%%%%%%%%
%%%%%%%%%%%%%%%%%%%%%%%%%%%%
\subsection{Contraction, with trivial fixed point}

In spite of the recurrence \eqref{E_rec}, 
the results in \cite{NR04a} do not apply. 
The issue is that 
the mean and variance
of $H_n$ are of poly-logarithmic order;
see Section \ref{S_moments} below. 
This leads to a trivial  fixed point equation 
$X\overset{d}{=}X$, 
yielding no information about $X$. 
Fortunately, the follow-up article \cite{NR04b}
deals with 
this very situation. 

Suppose that a sequence $(X_n)$ of random variables
satisfies
\begin{equation}\label{E_NRb}
X_n
\overset{d}{=} 
X_{I_n} +b_n, 
\end{equation}
where $(I_n,b_n)$ and $(X_n)$ are independent, 
and $I_n$ takes values
in $\{1,\ldots,n-1\}$.

Let $\mu_n=\E(X_n)$ and $\sigma^2_n={\rm Var}(X_n)$. 

As usual, we let $\|Y\|_p=(\E|Y|^p)^{1/p}$ 
denote the $L_p$-norm
of an $\R$-valued random variable $Y$. 

\begin{theorem}[Neininger and R\"{u}schendorf \cite{NR04b}]
\label{T_NR04b}
Suppose that $(X_n)$ satisfies \eqref{E_NRb}, with 
$\|X_n\|_3<\infty$, 
\begin{equation}\label{E_NR04b}
\limsup_{n\to\infty}
\E \log(I_n/n)<0,
\quad\quad 
\sup_{n\ge1}
\| \log(I_n/n)\|_3<\infty.
\end{equation}
Assume also that, 
for some $\alpha,\lambda,\kappa\in\R$
with $0\le\lambda<2\alpha$, 
and some $C>0$, we have that 
\begin{equation}\label{E_NR04b_params}
\|
b_n-\mu_n+\mu_{I_n}
\|_3=O(\log^\kappa n),
\quad\quad
\sigma^2_n=A_*\log^{2\alpha}n+O(\log^\lambda n), 
\end{equation}
and 
\[
\beta=
\min\{
3/2,3(\alpha-\kappa),3(\alpha-\lambda/2),\alpha-\kappa+1
\}>1.
\]
Then, as $n\to\infty$, 
\[
\frac{X_n-\mu_n}{\sqrt{A_*}\log^\alpha n}
\]
is asymptotically standard normal
distributed. 
\end{theorem}

Furthermore, it is shown that the 
distance between $X_n^*=(X_n-\mu_n)/\sigma_n$ 
and a standard normal 
random variable 
$Z$
is $O(1/\log^{\beta-1}n)$, with respect to the 
{\it Zolotarev metric} $\zeta_3$. 
As discussed in \cite{NR04b} (cf.\ \cite{Zol76,Zol77}),
for $\R$-valued random variables $U$ and $V$, 
\[
\zeta_3(U,V)=\sup|\E f(U)-\E f(V)|, 
\]
where the supremum is over all twice differentiable 
$f$, with 1-Lipschitz $f''$. Convergence with respect to 
$\zeta_3$ implies weak convergence.

%%%%%%%%%%%%%%%%%%%%%%%%%%%%
%%%%%%%%%%%%%%%%%%%%%%%%%%%%
%%%%%%%%%%%%%%%%%%%%%%%%%%%%
%%%%%%%%%%%%%%%%%%%%%%%%%%%%
%%%%%%%%%%%%%%%%%%%%%%%%%%%%
\section{Asymptotic normality of $H_n$}

We will prove the following result. 

Throughout this section, we let
$\mu_n=\E(H_n)$ and $\sigma^2_n={\rm Var}(H_n)$. 
Recall that $\zeta_3$ is the Zolotarev metric, 
as discussed above. 

\begin{theorem} 
\label{T_CLT2}
Let $H_n$ be the height of a uniformly 
random leaf in the critical beta-splitting tree $\cT_n$. 
Then
\[
H_n^*=(H_n-\mu_n)/\sigma_n
\]
is asymptotically standard normal distributed. 
Furthermore, for any $\eps>0$, 
\[
\zeta_3(H_n^*,Z)
=O(\log^{-1/2+\eps} n),
\]
where $Z$ is a standard normal random variable. 
\end{theorem}

%%%%%%%%%%%%%%%%%%%%%%%%%%%%
%%%%%%%%%%%%%%%%%%%%%%%%%%%%
%%%%%%%%%%%%%%%%%%%%%%%%%%%%
%%%%%%%%%%%%%%%%%%%%%%%%%%%%
%%%%%%%%%%%%%%%%%%%%%%%%%%%%
\subsection{A borderline case}\label{S_border}

In proving Theorem \ref{T_CLT2}, we cannot 
apply Theorem \ref{T_NR04b} directly. 
We have the distributional recurrence 
\eqref{E_rec}, however, 
Theorem \ref{T_NR04b} does not apply as stated 
because, 
by \eqref{E_vartheta} and \eqref{E_In}, 
\begin{align}\label{E_In/n}
\E\log^k(n/I_n)
&=
\frac{1}{\vartheta (n-1)}
\sum_{i=1}^{n-1} \frac{\log^k(n/i)}{n-i}\nonumber \\
&\sim\frac{1}{\log n} \int_0^1\frac{\log^k(1/x)}{1-x}dx
=\frac{k!\zeta(k+1)}{\log n}.
\end{align}
Therefore, 
\[
\lim_{n\to\infty}
\E \log(I_n/n)=0,
\]
and so the first condition in \eqref{E_NR04b} fails. 

As we will see, the precise second order term 
in the variance $\sigma^2_n$ (that is, identifying the constant 
$B_*$ in front of $\log^\lambda n$ in \eqref{E_NR04b_params} above) 
will play a crucial role in such a borderline case.

%%%%%%%%%%%%%%%%%%%%%%%%%%%%
%%%%%%%%%%%%%%%%%%%%%%%%%%%%
%%%%%%%%%%%%%%%%%%%%%%%%%%%%
%%%%%%%%%%%%%%%%%%%%%%%%%%%%
%%%%%%%%%%%%%%%%%%%%%%%%%%%%
\subsection{Bounding the moments}
\label{S_moments}

Aldous and Pittel 
\cite[Theorem 1.2]{AP23}
showed that 
\begin{align}
\mu_n&=A\log^2n
+B\log n+O(1),\label{E_EHn} \\
\sigma^2_n&=A_*\log^3n+O(\log^2 n),\label{E_VarHn}
\end{align}
where 
\begin{equation}\label{E_AB}
A = \frac{1}{2\zeta(2)},\quad 
B = \frac{\gamma}{\zeta(2)}
+\frac{\zeta(3)}{\zeta(2)^2},\quad
A_* = \frac{2\zeta(3)}{3\zeta(2)^3}. 
\end{equation}
and 
\begin{equation}\label{E_gamma}
\gamma
=\lim_{n\to\infty}(\vartheta(n)-\log n)
\end{equation}
is the {\it Euler--Mascheroni constant}. 

By \eqref{E_rec}, \eqref{E_EHn} and \eqref{E_VarHn}, we have 
(in the notation of 
Theorem \ref{T_NR04b}) that
$b_n=1$, $\kappa=2/3$, $\alpha=3/2$ and 
$\lambda=2$. 
In particular, to see that $\kappa=2/3$, let us note 
the following. 

\begin{lemma}
We have that 
\[
\|1-\mu_n+\mu_{I_n}\|_3
=O(\log^{2/3}n). 
\]
\end{lemma}

\begin{proof}
Using $x^2-y^2=(x-y)^2+2y(x-y)$, we have 
\begin{equation}\label{E_log2Inn}
\log^2(I_n)-\log^2 n=\log^2(I_n/n)+2\log n\log(I_n/n). 
\end{equation}
Hence, by \eqref{E_EHn}, 
\begin{equation}\label{E_kappa}
|1-\mu_n+\mu_{I_n}|
=O(1+\log n \log(n/I_n)).
\end{equation}
Therefore, by \eqref{E_In/n}, 
\[
\E|1-\mu_n+\mu_{I_n}|^3=O(\log^2 n), 
\]
and the claim follows. 
\end{proof}

As we will see, the $O(\log^2 n)$ term in
\eqref{E_VarHn} will not be sufficient in the proof of Theorem \ref{T_CLT2}
(specifically, 
when bounding $\E |\tau_n^2-G_n^2|^{3/2}$ 
in the proof of Claim \ref{C_beta52} below). 

More recently, Aldous and Janson \cite[Theorem 1.2]{AJ24b} have
obtained a detailed expansion for $\mu_n$. 
In particular, their results imply that 
\begin{equation}\label{E_AJ24}
\mu_n 
=
A\log^2n
+B\log n
+C+O(\log n /n),
\end{equation}
where
\begin{equation}\label{E_C}
C=\frac{1}{10}
+\frac{\gamma^2}{2\zeta(2)}
+\frac{\gamma\zeta(3)}{\zeta(2)^2}
+\frac{\zeta(3)^2}{\zeta(2)^3}.
\end{equation}

Combining this with arguments in 
the proof of Proposition 2.10 in \cite{AP23}, 
we obtain the following 
sharper estimate for the variance $\sigma_n^2$.

\begin{lemma}\label{L_varHn2}
We have that 
\begin{equation}\label{E_varHn2}
\sigma_n^2
=
A_*\log^3n+B_*\log^2 n+O(\log n),
\end{equation}
where
\begin{equation}\label{E_Bstar}
B_* 
=
-\frac{1}{2\zeta(2)}
-\frac{3\zeta(4)}{\zeta(2)^3}
+\frac{2\gamma\zeta(3)}{\zeta(2)^3}
+\frac{4\zeta(3)^2}{\zeta(2)^4}.
\end{equation}
\end{lemma}

We note that 
the specific values of 
$A,B,C$ and 
$A_*,B_*$ will not play a special role in the proof 
of Theorem \ref{T_CLT2}. 
We need only that they exist. 

\begin{proof}
Put 
\[
W_n = \sigma_n^2 - A_*\log^3n. 
\]
The bound \eqref{E_VarHn} is proved in 
\cite{AP23} as follows: The mean $\mu_n$ satisfies
(see (43) in \cite{AP23})
\[
\mu_n 
=
\frac{1}{\vartheta(n-1)}\sum_{k=1}^{n-1}\frac{\mu_k}{n-k}
+1.
\]
The proof of Proposition 2.10 in \cite{AP23}
shows that $W_n$ satisfies 
\begin{equation}\label{E_Wn}
W_n 
=
\frac{1}{\vartheta(n-1)}\sum_{k=1}^{n-1}\frac{W_k}{n-k}
+O(1),
\end{equation}
and thereby concludes that $W_n = O(\mu_n)=O(\log^2 n)$, 
using \eqref{E_EHn}. The proof of \eqref{E_Wn} is based on the 
fact (see (45) in \cite{AP23}) that 
\begin{equation}\label{E_sigRec}
\sigma_n^2 = -1+\frac{1}{\vartheta(n-1)}
\sum_{k=1}^{n-1}\frac{\sigma_k^2+(\mu_n-\mu_k)^2}{n-k}.
\end{equation}

To sharpen \eqref{E_VarHn} to \eqref{E_varHn2},
we will use \eqref{E_AJ24}
to show that 
\begin{equation}\label{E_X}
W_n 
=
\frac{1}{\vartheta(n-1)}\sum_{k=1}^{n-1}\frac{W_k}{n-k}
+X+O(\log^{-1} n),
\end{equation}
where $X$ is a constant, to be determined
in \eqref{E_defX} below.

The $O(1)$ term 
in \eqref{E_Wn} arises in \cite{AP23}
when estimating the sums
\[
\sum_{k=1}^{n-1}\frac{\log^3  k}{n-k},
\quad\quad 
\sum_{k=1}^{n-1}
\frac{(\mu_n - \mu_k)^2}{n-k}.
\] 
Using \eqref{E_AJ24}, 
and other estimates in \cite{AP23}, 
we will show in Appendix \ref{A_sums}
 that 
\begin{align}
\frac{1}{\vartheta(n-1)}
\sum_{k=1}^{n-1}\frac{\log^3  k}{n-k}
&=\log^3 n 
-3\zeta(2)\log n
+O(1/\log n)\nonumber \\
&\quad\quad 
+3\gamma\zeta(2)
+6\zeta(3)\label{E_log3}
\end{align}
and 
\begin{align}
\frac{1}{\vartheta(n-1)}\sum_{k=1}^{n-1}
\frac{(\mu_n - \mu_k)^2}{n-k}
&=8A^2\zeta(3)\log n+O(1/\log n)\nonumber\\
&\quad\quad-8A(A\gamma-B)\zeta(3)
-24A^2\zeta(4).\label{E_munk}
\end{align}

Substituting \eqref{E_munk} into \eqref{E_sigRec}, we find that 
\begin{align}
\sigma_n^2 
&=
\frac{1}{\vartheta(n-1)}\sum_{k=1}^{n-1}\frac{\sigma_k^2}{n-k}
+8A^2\zeta(3)\log n+O(\log^{-1} n)\nonumber\\ 
&\quad\quad
-1
-8A(A\gamma-B)\zeta(3)
-24A^2\zeta(4).\label{E_recvar}
\end{align}
Finally, combining \eqref{E_log3} and \eqref{E_recvar}
(and noting that $8A^2\zeta(3)=3\zeta(2)A_*$ by \eqref{E_AB},
so that the $\log n$ terms cancel),
we obtain \eqref{E_X} above, with 
\begin{equation}\label{E_defX}
X =
-1
-8A(A\gamma-B)\zeta(3)
-24A^2\zeta(4)
+A_*(3\gamma\zeta(2)
+6\zeta(3)).
\end{equation}
It follows that 
\[
W_n
=(X+O(1/\log n))\mu_n
=AX\log^2 n+O(\log n).
\]
Therefore, we find that \eqref{E_varHn2} holds with 
$B_*=AX$, which by \eqref{E_AB} 
simplifies to \eqref{E_Bstar} above, 
as claimed.  
\end{proof}

It should be possible 
(cf.\ \cite[Remark 1.7]{AJ24b})
to obtain further terms in 
the expansion for $\sigma_n^2$, 
by being more careful with the 
estimates \eqref{E_log3}
and \eqref{E_munk}, 
but \eqref{E_varHn2}
will suffice for our purposes.

%%%%%%%%%%%%%%%%%%%%%%%%%%%%
%%%%%%%%%%%%%%%%%%%%%%%%%%%%
%%%%%%%%%%%%%%%%%%%%%%%%%%%%
%%%%%%%%%%%%%%%%%%%%%%%%%%%%
%%%%%%%%%%%%%%%%%%%%%%%%%%%%
\subsection{The proof}
\label{S_proof}

In what follows, we will assume familiarity with the proof
of Theorem 2.1 in \cite{NR04b}, 
and the notation introduced therein. 
We will not explain the full proof given in \cite{NR04b}, 
but rather only discuss the few 
places that need adjustment. 

\begin{proof}[Proof of Theorem \ref{T_CLT2}]

Specifically, there are two main parts in 
the proof of Theorem 2.1 in \cite{NR04b}
that require attention. The first is 
Lemma 3.1
\cite{NR04b}. 
In fact, as we will see, 
the proof of this result simplifies in our special case. 
Secondly, we will revisit the upper bound
\cite[(19)]{NR04b}, as this 
estimate is used in the inductive proof of 
Lemma 3.1 in \cite{NR04b}. 

We start with the second part. 
Recall that $\P(I_n\in\{0,n\})=0$. 
We set 
$\delta=1$, and let 
$\ell_n=\log n+\1_{n=1}$ play the role of 
$L_\delta(n)$.

As noted in Section \ref{S_moments} above, 
we have that $\alpha=3/2$. 
In particular, we have 
(in the notation of  \cite{NR04b}) 
that 
\[
b^{(n)}=\frac{1-\mu_n+\mu_{I_n}}{\sqrt{A_*}\ell_n^{3/2} },\quad\quad 
\tau_n=\frac{\sigma_n}{\sqrt{A_*}\ell_n^{3/2}},\quad\quad 
G_n=\frac{\sigma_{I_n}}{\sqrt{A_*}\ell_n^{3/2}}.
\]

A crucial step in 
the proof of Theorem 2.1 in \cite{NR04b}
is bounding  
$\zeta_3(Z_n^*,N_n)$, as this corresponds to 
$r_n$ in Lemma 3.1 in \cite{NR04b}. 
For the specific case of the critical beta-splitting tree, 
we will show the following. 

\begin{claim}
\label{C_beta52}
We have $\zeta_3(Z_n^*,N_n)=O(\log^{-5/2}n)$.
\end{claim}

\begin{proof}
A general upper bound 
for $\zeta_3(Z_n^*,N_n)$
is obtained at (23) in \cite{NR04b},
based on the bounds for $|\E S_2|$
and $\E|R(\hat Z_n^*,N)|$. Using
these, it follows that, up to multiplicative constants,  $\zeta_3(Z_n^*,N_n)$ is
bounded above by the expected value of 
\begin{equation}\label{E_exp52}
|b^{(n)}|^3
+|\tau_n-1|^3
+|G_n-1|^3
+|b^{(n)}(\tau_n-1)|
+|b^{(n)}(G_n-1)|
+|\tau_n^2-G_n^2|^{3/2}. 
\end{equation}

We proceed term by term. 
By \eqref{E_In/n} and \eqref{E_kappa}, 
\[
\E|b^{(n)}|^3
=O( \log^{-5/2} n). 
\]
By \eqref{E_VarHn}, 
\begin{equation}\label{E_taun}
|\tau_n-1|
\le\frac{|\sigma^2_n-A_*\log^3 n|}{A_*\log^3 n} 
=O(\log^{-1} n)
\end{equation}
Note that, using $x^3-y^3=(x-y)^3+3xy(x-y)$, we have 
\begin{equation}\label{E_log3Inn}
\log^3(I_n)-\log^3 n=\log^3(I_n/n)+3\log n\log(I_n)\log(I_n/n), 
\end{equation}
It follows by \eqref{E_VarHn} that 
\begin{equation}\label{E_Gn}
|G_n-1|
\le\frac{|\sigma^2_{I_n}-A_*\log^3 n|}{A_*\log^3 n} 
=O\left(\frac{|\log(I_n/n)|+1}{\log n}\right).
\end{equation}
Therefore, using \eqref{E_In/n}, 
\[
\E|G_n-1|^3=O(\log^{-3} n). 
\]
By \eqref{E_In/n}, \eqref{E_kappa} and \eqref{E_taun},
\[
\E |b^{(n)}(\tau_n-1)|
=O\left(
\frac{1}{\log^{5/2}n}
+\frac{\E|\log(I_n/n)|}{\log^{3/2}n}
\right)
=O(\log^{-5/2}n).
\]
Likewise, by \eqref{E_kappa} and \eqref{E_Gn},
\[
|b^{(n)}(G_n-1)|\\
=O\left(\frac{|\log(I_n/n)|+1}{\log^{5/2}n}
+\frac{|\log(I_n/n)|+|\log(I_n/n)|^2}{\log^{3/2}n}\right),
\]
and so, by \eqref{E_In/n},
\[
\E |b^{(n)}(G_n-1)|
=  O(\log^{-5/2}n).
\]
Finally, using \eqref{E_log2Inn}, 
\eqref{E_varHn2}
and \eqref{E_log3Inn}, 
we observe that 
\[
|\tau_n^2-G_n^2|
=
O\left(\frac{1}{\log^2 n}
+\frac{|\log(I_n/n)|}{\log n}\right).
\]
Therefore, by \eqref{E_In/n}, and using the fact that 
$(1+x)^{3/2}\le 1+(3/2)x+x^{3/2}$ for all $x\ge0$, 
it follows that 
\[
\E |\tau_n^2-G_n^2|^{3/2}
\le O\left(\frac{\E(1+\log n|\log(I_n/n)|)^{3/2}}{\log^3 n}
\right)
=O(\log^{- 5/2} n). 
\]
(Note that using \eqref{E_VarHn} instead of \eqref{E_varHn2}
yields only
$O(\log^{- 3/2} n)$, 
which is not sufficient for our purposes.)

Altogether, we conclude that the expected value of \eqref{E_exp52},
and so also $\zeta_3(Z_n^*,N_n)$, is 
 $O(1/\log^{5/2})$,
as claimed. 
\end{proof}

Next, we 
prove the following analogue of 
Lemma 3.1 in 
\cite{NR04b}. 

\begin{claim}
\label{C_ind}
Let $I_n$ be as in \eqref{E_In}. Suppose that 
nonnegative sequences $(d_n)$ and $(r_n)$ satisfy 
\begin{equation}\label{E_dn}
d_n\le 
\E\left[\left(\frac{\ell_{I_n}}{\ell_n}\right)^{9/2} d_{I_n}\right] 
+r_n,\quad\quad n\ge2,
\end{equation}
and 
\begin{equation}\label{E_rn}
r_n=O(\log^{-5/2} n).
\end{equation}
Then, for all small $\eps>0$, it follows that 
\[
d_n=O(\log^{-1/2+\eps}n).
\]
\end{claim}

\begin{proof}
To see this, we will follow the proof of \cite[Lemma 3.1]{NR04b}.
We can, in fact, make some simplifications in this special case. 
Using \eqref{E_In}, \eqref{E_In/n} and \eqref{E_rn}, 
let $M>0$ and $n_1$ be such that 
$r_n
\le M/\log^{5/2} n$
and 
\[
\frac{\E[\log (I_n/n)]}{\log n}
+\frac{\P(I_n=1)}{\log n}
+\frac{1}{\log^{2+\eps} n}
\le 0
\]
for all $n\ge n_1$. 

Put
\[
R=M\vee\max\{d_k\ell^{1/2-\eps}_k:1\le k\le n_1\}.
\]
To prove the claim, we will show, by induction, that 
$d_n\le R/\ell_n^{1/2-\eps}$. 
By the choice of $R$, there is nothing to prove for $n\le n_1$. 
On the other hand, for $n> n_1$, by 
\eqref{E_dn}, the choice of $n_1$, and 
the inductive hypothesis, 
\begin{align*}
d_n
&\le 
\E\left[\left(\frac{\ell_{I_n}}{\ell_n}\right)^{9/2}
\frac{R}{\ell^{1/2-\eps}_{I_n}}
\right]
+ \frac{M}{\log^{5/2} n}
\le\frac{R}{\ell^{1/2-\eps}_n}
\left[\E\left(\frac{\ell_{I_n}}{\ell_n}\right)
+\frac{1}{\log^{2+\eps} n}
\right]\\
&\le\frac{R}{\ell^{1/2-\eps}_n}
\left(1+
\frac{\E[\log (I_n/n)]}{\log n}
+\frac{\P(I_n=1)}{\log n}
+\frac{1}{\log^{2+\eps} n}
\right)
\le \frac{R}{\ell^{1/2-\eps}_n}, 
\end{align*}
as required. 
\end{proof}

This 
finishes the proof, 
as the rest of the proof of 
Theorem 2.1 in 
\cite{NR04b}
applies, without any further changes. 
Using Claim \ref{C_beta52}, we apply 
Claim \ref{C_ind} to conclude that 
$\zeta_3(H_n^*,Z)
=O(\log^{- 1/2+\eps} n)$.  
\end{proof}

%%%%%%%%%%%%%%%%%%%%%%%%%%%%
%%%%%%%%%%%%%%%%%%%%%%%%%%%%
%%%%%%%%%%%%%%%%%%%%%%%%%%%%
%%%%%%%%%%%%%%%%%%%%%%%%%%%%
%%%%%%%%%%%%%%%%%%%%%%%%%%%%
\section{Final remarks}

%%%%%%%%%%%%%%%%%%%%%%%%%%%%
%%%%%%%%%%%%%%%%%%%%%%%%%%%%
%%%%%%%%%%%%%%%%%%%%%%%%%%%%
%%%%%%%%%%%%%%%%%%%%%%%%%%%%
%%%%%%%%%%%%%%%%%%%%%%%%%%%%
\subsection{Splitting at exponential times}
\label{S_ctsT}

In \cite{AP23}, a variation of 
the critical beta-splitting tree is also considered, 
in which  
splitting events occur continuously in time. 
Specifically, subsets have 
exponential holding times, 
with rates $\vartheta(k-1)$
on subsets of size $k$. 

More precisely, recall the construction of
$\cT_n$ in Section \ref{S_intro}. 
When a set $S$ is split into subsets $S_1,S_2$
two edges are drawn from $v(S)$ to its children $v(S_1)$
and $v(S_2)$. In the discrete model, 
both of these edges have  
length 1. In the continuous model, the length of 
these edges are determined by an exponential 
random variable with rate $\vartheta(\#S-1)$.

In \cite{AP23}, a central limit theorem is proved 
for the {\it time-height} $\hat H_n$ of a random leaf. 
See also \cite{AJ23,AJ24b} 
for alternative 
proofs. 

We note that $\hat H_n$ satisfies a similar 
distributional recurrence. 
Specifically, instead of \eqref{E_rec}, we have 
\[
\hat H_n
\overset{d}{=} 
\hat H_{I_n}
+t_n,
\]
where $t_n$ is exponential with rate $\vartheta(n-1)$,  
and $I_n$ is as before. In other words, we simply 
replace $1$ with $t_n$ in \eqref{E_rec}. 

Estimates for the leading order terms in 
$\hat\mu_n = \E(\hat H_n)$ and 
$\hat\sigma_n^2={\rm Var}(\hat H_n)$
are given in \cite{AP23}. More recently, 
an asymptotic expansion for $\hat\mu_n$
has been obtained \cite[Theorem 1.1]{AJ24b}. 
In particular, we have that 
\[
\hat\mu_n = \hat A\log n +\hat B+O(1/n),
\]
where
\[
\hat A=\frac{1}{\zeta(2)},\quad\quad
\hat B=\frac{\gamma}{\zeta(2)}
+\frac{\zeta(3)}{\zeta(2)^2}.
\]
Moreover, by \cite[Theorem 1.3]{AJ24b}, 
\[
\hat\sigma_n^2
=\hat A_*\log n +\hat B_* +O(\log n/n),
\]
where
\[
\hat A_*=\frac{2\zeta(3)}{\zeta(2)^2},\quad\quad
\hat B_*=
-\frac{3}{5 \zeta(2)}
+\frac{2\gamma\zeta(3)}{\zeta(2)^3}
+\frac{5\zeta(3)^2}{\zeta(2)^4}. 
\]

The proof of 
Theorem \ref{T_CLT2} for $H_n$ can be adapted 
to  the case $\hat H_n$ as follows: 
First, we note that $1$ is replaced by $t_n$ in $b^{(n)}$. 
Next, in all of $b^{(n)}$, $\tau_n$ and $G_n$, 
we replace $A_*$ with $\hat A_*$ and $\ell_n^{3/2}$
with $\ell_n^{1/2}$.  
Then, 
along similar lines (we omit the details), it can be shown that 
the expected values of all quantities in \eqref{E_exp52}
are bounded by $O( \log^{- 5/2})$, as before. 
Therefore, by essentially the same proof, we obtain the following result.  

\begin{theorem} 
\label{T_CLThatH}
Let $\hat H_n$ be the time-height of a uniformly 
random leaf in the critical beta-splitting tree $\cT_n$. 
Then
\[
\hat H_n^*=(\hat H_n-\hat \mu_n)/\hat \sigma_n
\]
is asymptotically standard normal distributed. 
Furthermore, for any $\eps>0$, 
\[
\zeta_3(\hat H_n^*,Z)
=O(\log^{- 1/2+\eps} n),
\]
where $Z$ is a standard normal random variable. 
\end{theorem}

%%%%%%%%%%%%%%%%%%%%%%%%%%%%
%%%%%%%%%%%%%%%%%%%%%%%%%%%%
%%%%%%%%%%%%%%%%%%%%%%%%%%%%
%%%%%%%%%%%%%%%%%%%%%%%%%%%%
%%%%%%%%%%%%%%%%%%%%%%%%%%%%
\subsection{Collisions in the $\beta(2,b)$-coalescent}
\label{S_Collisions}

Finally, as mentioned in Section \ref{S_purpose} above, 
let us discuss
the central limit theorem proved 
in \cite{IMM09}
for the number of collisions $X_n$ 
in the {\it $\beta(2,b)$-coalescent}. 
See, e.g., Pitman \cite{Pitman99}, Sagitov \cite{Sag99}
and  Gnedin, Iksanov and Marynych \cite{GIM14}
for background. 

In \cite[(2)]{IMM09}, there is a similar recurrence 
as \eqref{E_rec} above. Also, 
for similar reasons as in the current article 
(compare  \eqref{E_In/n} 
with \cite[Remark 3.2]{IMM09}), 
Theorem \ref{T_NR04b} 
does not apply as stated. 
In \cite[(14)]{IMM09}, 
an alternative, more complicated, 
recurrence is derived for which  
Theorem \ref{T_NR04b} applies. 
However, the authors
ask \cite[Remark 1.6]{IMM09}
if a more direct proof is possible, 
using only the original recursion \cite[(2)]{IMM09}.  

By \cite[Theorem 1.1]{IMM09}, 
\begin{align}\label{E_mCn}
\E(X_n)&=a\log^2 n+b\log n+O(1),\nonumber \\
{\rm Var}(X_n)&=a_*\log^3 n+O(\log^2n),
\end{align}
for explicit constants $a$, $b$ and $a_*$
(cf.\ \eqref{E_EHn} and \eqref{E_VarHn} above). 
Hence, if the $O(\log^2n)$ term in the expansion for 
${\rm Var}(X_n)$ can be replaced with some
$b_*\log^2n +O(\log n)$, 
the arguments in the current article
would give an answer to the question 
in \cite[Remark 1.6]{IMM09}. 

Martin M\"{o}hle has informed us, by private
communication, that the existence of $b_*$ 
is determined by whether 
$v_n$ in \cite[Lemma A.2]{IMM09} oscillates or converges. 
See Proposition 2.1 in M\"{o}hle \cite{Moh14}
for a related result. This result, however, does not apply
due to a lack of information about hitting probabilities 
in the $\beta(2,b)$-coalescent
block counting process. 
Finding $b_*$ (if it exists) might be an interesting 
question in its own right.

%%%%%%%%%%%%%%%%%%%%%%%%%%%%
%%%%%%%%%%%%%%%%%%%%%%%%%%%%
%%%%%%%%%%%%%%%%%%%%%%%%%%%%
%%%%%%%%%%%%%%%%%%%%%%%%%%%%
%%%%%%%%%%%%%%%%%%%%%%%%%%%%
\appendix

%%%%%%%%%%%%%%%%%%%%%%%%%%%%
%%%%%%%%%%%%%%%%%%%%%%%%%%%%
%%%%%%%%%%%%%%%%%%%%%%%%%%%%
%%%%%%%%%%%%%%%%%%%%%%%%%%%%
%%%%%%%%%%%%%%%%%%%%%%%%%%%%
\section{Technical estimates}\label{A_sums}

In this section, we prove
\eqref{E_log3} and \eqref{E_munk}
used in the proof 
of Lemma \ref{L_varHn2}.

%%%%%%%%%%%%%%%%%%%%%%%%%%%%
%%%%%%%%%%%%%%%%%%%%%%%%%%%%
%%%%%%%%%%%%%%%%%%%%%%%%%%%%
%%%%%%%%%%%%%%%%%%%%%%%%%%%%
%%%%%%%%%%%%%%%%%%%%%%%%%%%%
\subsection{Estimate \eqref{E_log3}}

\begin{proof}[Proof of \eqref{E_log3}]
In the proof of Proposition 2.10 in \cite{AP23},
it is observed that 
\begin{align}
\frac{1}{\vartheta(n-1)}
\sum_{k=1}^{n-1}\frac{\log^3  k}{n-k}
&=\log^3 n 
+\frac{3\log^2 n}{\vartheta(n-1)}
\sum_{k=1}^{n-1}\frac{\log(k/n)}{n-k}\nonumber\\
&\quad
+\frac{3\log n}{\vartheta(n-1)}
\sum_{k=1}^{n-1}\frac{\log^2 (k/n)}{n-k}
+\frac{1}{\vartheta(n-1)}
\sum_{k=1}^{n-1}\frac{\log^3 (k/n)}{n-k}.\label{E_prop210} 
\end{align}

As noted in \cite{AP23},
\[
\log n=
\vartheta(n-1)-\gamma+O(1/n). 
\]
In particular, 
\begin{equation}\label{E_log/theta}
\frac{\log n}{\vartheta(n-1)}
=1-\frac{\gamma}{\log n}+O(\log^{-2} n). 
\end{equation}

Lemma 2.2 in \cite{AP23} implies that 
\begin{equation}\label{E_log1}
\sum_{k=1}^{n-1}\frac{\log(k/n)}{n-k}
=-\zeta(2)+O(1/n).
\end{equation}
Furthermore, in the proof of Proposition 2.5 in \cite{AP23}
it is observed  
(using Euler's summation formula; see (11) in \cite{AP23}) 
that 
\begin{equation}\label{E_log2}
\sum_{k=1}^{n-1}\frac{\log^2(k/n)}{n-k}
=2\zeta(3)+O(\log^2 n/n).
\end{equation}
More generally, 
\begin{equation}\label{E_logr}
\sum_{k=1}^{n-1}\frac{\log^r (k/n)}{n-k}
=(-1)^rr!\zeta(r+1)+O(\log^r n/n),
\end{equation}
for $r\ge1$.

Combining \eqref{E_prop210}--\eqref{E_logr},
we obtain \eqref{E_log3}. 
\end{proof}

%%%%%%%%%%%%%%%%%%%%%%%%%%%%
%%%%%%%%%%%%%%%%%%%%%%%%%%%%
%%%%%%%%%%%%%%%%%%%%%%%%%%%%
%%%%%%%%%%%%%%%%%%%%%%%%%%%%
%%%%%%%%%%%%%%%%%%%%%%%%%%%%
\subsection{Estimate \eqref{E_munk}}

\begin{proof}[Proof of \eqref{E_munk}]
Using 
\eqref{E_AJ24}, let us write 
\[
\mu_n 
=
A\log^2n
+B\log n
+C+\xi(n),
\]
where
$\xi(n)=O(\log n /n)$. 
Then  
\begin{align}
(\mu_n - \mu_k)^2
&=[(A\log(k/n)+2A\log n+B)\log(k/n)+\xi(k)-\xi(n)]^2\nonumber \\
&=4A^2\log^2(k/n) \log^2 n
+[4AB\log^2(k/n)+4A^2\log^3(k/n)]\log n\nonumber \\
&\quad+O[\log^2(k/n)+\log^3(k/n)+\log^4(k/n)]\nonumber \\
&\quad+O[\xi(k)(\xi(k)+\log(k/n)\log n)]. \label{E_munkexp}
\end{align}

As noted
below (21) in \cite{AP23}, we have
\[
\sum_{k=1}^{n-1}\frac{1}{k^2(n-k)}=O(1/n)
\]
and
\[
\sum_{k=1}^{n-1}\frac{\log(n/k)}{k(n-k)}=O(\log^2 n/n).
\]
Therefore
\begin{equation}\label{E_xi1}
\sum_{k=1}^{n-1}\frac{\xi(k)^2}{n-k}
\le O(\log^2 n)
\sum_{k=1}^{n-1}\frac{1}{k^2(n-k)}
=O(\log^2 n/n)
\end{equation}
and
\begin{equation}\label{E_xi2}
\sum_{k=1}^{n-1}\frac{\xi(k)\log(k/n)}{n-k}
\le O(\log n)\sum_{k=1}^{n-1}\frac{\log(k/n)}{k(n-k)}
=O(\log^3 n/n).
\end{equation}

Applying  
\eqref{E_log/theta} and 
\eqref{E_logr}--\eqref{E_xi2},
we obtain \eqref{E_munk}. 
\end{proof}

%%%%%%%%%%%%%%%%%%%%%%%%%%%%
%%%%%%%%%%%%%%%%%%%%%%%%%%%%
%%%%%%%%%%%%%%%%%%%%%%%%%%%%
%%%%%%%%%%%%%%%%%%%%%%%%%%%%
%%%%%%%%%%%%%%%%%%%%%%%%%%%%

\makeatletter
\renewcommand\@biblabel[1]{#1.}
\makeatother

\providecommand{\bysame}{\leavevmode\hbox to3em{\hrulefill}\thinspace}
\providecommand{\MR}{\relax\ifhmode\unskip\space\fi MR }
\providecommand{\MRhref}[2]{%
  \href{http://www.ams.org/mathscinet-getitem?mr=#1}{#2}
}
\providecommand{\href}[2]{#2}

%%%%%%%%%%%%%%%%%%%%%%%%%%%%
%%%%%%%%%%%%%%%%%%%%%%%%%%%%
%%%%%%%%%%%%%%%%%%%%%%%%%%%%
%%%%%%%%%%%%%%%%%%%%%%%%%%%%
%%%%%%%%%%%%%%%%%%%%%%%%%%%%
\end{document}